\documentclass[12pt]{article}
\usepackage[intlimits]{amsmath}
\usepackage{amsfonts,amssymb,amscd,amsthm,cite}
\input xy
\xyoption{all}

\usepackage{ifpdf}
\ifpdf 
  \usepackage[pdftex]{graphicx}
  \DeclareGraphicsExtensions{.pdf,.png,.jpg,.jpeg,.mps}
  \usepackage{pgf}
  \usepackage{tikz}
\else 
  \usepackage{graphicx}
  \DeclareGraphicsExtensions{.eps,.bmp}
  \usepackage{pgf}
  \usepackage{tikz}
  \usepackage{pstricks}
\fi

\def\re{\operatorname{Re}}

\newcommand{\Vect}{\mathop{\rm Vect}}

 \def\lra{\longrightarrow}

\newtheorem{theorem}{Theorem}[section]
\newtheorem{proposition}{Proposition}[section]
\newtheorem{lemma}{Lemma}[section]
\newtheorem{corollary}{Corollary}[section]

\theoremstyle{definition}

\newtheorem{definition}{Definition}[section]

\newtheorem{remark}{Remark}[section]

\title{Elliptic $G$-operators on manifolds with isolated singularities}

\author{A.Yu. Savin,  B.Yu. Sternin}

\date{}

\begin{document}

\maketitle

\begin{abstract}
In the present work we study elliptic operators on manifolds with singularities in the situation, when the manifold is endowed with an action of a discrete group $G$.  
As usual in elliptic theory, the Fredholm property of an operator is governed by the properties of its principal symbol.  We show that the principal symbol   in our situation is a pair, consisting of  the symbol on the main stratum (interior symbol) and the symbol at the conical point (conormal symbol).  Fredholm property of elliptic elements is obtained.
\end{abstract}

\tableofcontents

\section*{Introduction} 

In the present work we study elliptic operators on manifolds with singularities in the situation, when the manifold is endowed with an action of a discrete group $G$ of diffeomorphisms. More precisely, we consider operators equal to linear combinations of products of (pseudo)differential  operators and shift operators along the orbits of the group action. Such operators are called  $G$-{\em operators},  since they are associated with an action of group $G$. Let us mention that the theory of $G$-operators on smooth manifolds is presently well developed: Fredholm property and related analytic questions are studied (e.g., see \cite{AnLe1,AnLe2}), index formulas and related topolo
gical problems are considered (e.g., see \cite{NaSaSt17,SaSt40,SaSt39}).  In this paper we study $G$-operators on manifolds with isolated singularities of conical type. Of course, the results of the paper can be extended to operators on manifolds with higher singularities: edges, corners, etc.   (e.g., see \cite{NSScS99,NaSaSt5,NaSaSt6}).

As usual in elliptic theory, Fredholm property of an operator is governed by the properties of its principal symbol.  We show that the principal symbol of a $G$-operator in our situation is a pair, consisting of  the symbol on the main stratum (interior symbol) and the symbol at the conical point (conormal symbol). In addition, it turns out that (unlike the interior symbol) the conormal symbol {\em is not} an element of the crossed product and therefore the standard technique of proving the Fredholm property of $G$-operators, which relies on crossed products (see \cite{AnLe2}), does not apply. One more feature of our theory is the fact that the conormal symbol at the conical point is an operator family exponentially depending on the parameter (cf. the classical theory of parameter-dependent families   \cite{AgVi1}). 
Finally, let us note that in the proof of the Fredholm property we systematically use the theory of   $C^*$-algebras.  

Let us briefly describe the contents of the paper. First, we give the statement of the problem   (Section 1). Then we  formulate the finiteness theorem, which is the main result of the paper. To this end, we compute the symbol of $G$-operators  using the method of frozen coefficients. Here, since we deal with nonlocal operators (containing shift operators), one has to freeze the coefficients of the operator on the entire orbit, rather than at one point. We also consider two examples: 1)   $\mathbb{Z}$-operators on the sphere $\mathbb{S}^2$; and 2) $\mathbb{Z}^2$-operators on the half-line $\mathbb{R}_+$. In these examples, the conditions, which give the Fredholm property, are written out explicitly.  Then (Section 5) we give the proof of the Fredholm property using the theory of  $C^*$-algebras.  As an important technical tool, we construct the theory of elliptic parameter-dependent families of   $G$-operators. At the end of the paper, there is a small appendix, in which we give a direct computation of the interior symbol of $G$-operators using the method of frozen coefficients.  

Let us note that problems of this type were considered earlier in the works of A.B.~Antonevich and his coworkers (e.g., see \cite{AnLe2} and the references there), however, most of the results obtained on this topic dealt with operators acting on   $L^2$-spaces, whereas we consider operators in weighted Sobolev spaces    $H^{s,\gamma}$ (see \cite{Kon1}).

The work was partially supported by RFBR (projects Nr.~12-01-00577 and~15-01-08392), DAAD, and Simons Foundation. We are grateful to Prof. Elmar Schrohe and Institute of Analysis at the Leibniz University of Hannover for hospitality. 

\section{Statement of the problem} 

In this section, we formulate the main problem of this work concerning elliptic operators associated with an action of a discrete group on a manifold with conical singularities. For a detailed exposition of elliptic theory on manifolds with conical singularities (without group actions) we refer the reader to \cite{Kon1}, see also \cite{NSScS99}.

\paragraph{$G$-manifolds with  conical points.} 
Let $M$ be a manifold with a conical point\footnote{The results of this paper can be generalized in a standard way to manifolds with an arbitrary finite number of conical points.} denoted by $pt$.   This means that the complement $M\setminus pt$ is a smooth manifold and there exists a diffeomorphism of a punctured neighborhood of the conical point and the Cartesian product  
$$
  (0,1)\times \Omega \quad \text{ with coordinates }(r,\omega).
$$
Here $\Omega$ is a closed smooth manifold (called the base of the cone), while $r=0$ corresponds to the conical point (for this value of $r$    $\Omega$ collapses to point).

{\em A diffeomorphism $g:M\to M$}   is a mapping such that:
\begin{enumerate}
\item[1)] $g:M\to M$ is a homeomorphism, which preserves the strata   $M\setminus pt$ and  $pt$ and induces diffeomorphisms of these strata;
\item[2)] in a neighborhood of the conical point the diffeomorphism has the form  
\begin{equation}\label{eq-diff1}
g(r,\omega)=(r e^{\beta_g},g(\omega)),
\end{equation}
where  $\beta_g$ is a real number and the diffeomorphism of $\Omega$  is also denoted by $g$ for brevity. 
\end{enumerate}  

Let $G$ be a discrete group.
\begin{definition}
Manifold $M$ is a {\em $G$-manifold}, if it is endowed with   $G$-action.
\end{definition}
To simplify the notation, we shall identify elements   $g\in G$ and corresponding diffeomorphisms   $g:M\to M$.

\paragraph{Weighted Sobolev spaces.}
Recall the definition of weighted Sobolev space  $H^{s,\gamma}(M)$ of functions on a manifold with conical points. The elements of this space are functions in the Sobolev space   $H^s$ outside a neighborhood of the conical point. Hence, it suffices to define the norm only for functions supported in a small neighborhood of the conical point.   

The norm of function $u(r,\omega)$ is defined as
\begin{equation}\label{eq-mel1aa}
\|u\|^2_{s,\gamma}=\int_{\Omega\times\mathbb{R}_+}\left|\left(1+\left(ir\frac\partial{\partial r}\right)^2+\Delta_{\Omega}\right)^{s}(r^{ \gamma}u(r,\omega))\right|^2_{L^2(\Omega)}d\omega \frac {dr}r,
\end{equation}
where $\Delta_\Omega$ is a nonnegative Laplace operator on  $\Omega$, or equivalently
\begin{equation}\label{eq-mel1}
\|u\|^2_{s,\gamma}=\int_{L_\gamma}\left\|\left(1-p^2+\Delta_{\Omega}\right)^{s}\widetilde u(p,\omega)\right\|^2_{L^2(\Omega)}dp,
\end{equation}
where 
\begin{equation}\label{eq-mel2}
\widetilde u(p,\omega)=\mathcal{M}_{r\to p}u=\frac 1{\sqrt{2\pi i}}\int_0^\infty r^p u(r,\omega) \frac {dr}r 
\end{equation}
is the Mellin transform of function $u(r,\omega)$ with respect to the radial variable $r$; 
$$
 L_\gamma=\{p\in\mathbb{C}\;|\;\re p=\gamma\}
$$ 
is the weight line.

\begin{remark}
The space $H^{0,\gamma}(M)$ is  just the  space $L^2\left(M,\mu_\gamma\right)$ of  square-integrable functions with respect to the measure 
$$
  \mu_\gamma= r^{ 2\gamma} \mu \frac{dr}r,
$$
where $\mu$ is a measure on $\Omega$. The equivalence of expressions   \eqref{eq-mel1} and  \eqref{eq-mel1aa} follows from the properties of the Mellin transform, in particular, from Mellin--Plancherel equality  (e.g., see \cite{Ster8}):
$$
 \int_{\re p=\gamma} |\widetilde f(p)|^2dp= \int_0^{\infty}r^{2\gamma} |f(r)|^2\frac {dr}r.
$$
\end{remark}

\paragraph{$G$-operators on manifolds with conical points.}
Let $M$ be a  $G$-manifold. The group action induces a representation of $G$ in Sobolev spaces by shift operators:
$$
T_g: H^{s,\gamma}(M)\longrightarrow H^{s,\gamma}(M),\text{ where } (T_g u)(x)=u(g^{-1}x).
$$
\begin{definition}
{\em A conical  $G$-operator} on $M$ is an operator of the form
\begin{equation}\label{eq-op12a}
 D=\sum_{h\in G} D_h T_h:H^{s,\gamma}(M)\longrightarrow H^{s-m,\gamma+m}(M),
\end{equation}  
where
$$
D_h=r^{-m}D_h\left(r,-r\frac{\partial}{\partial r},\omega,-i\frac{\partial}{\partial \omega}\right)
$$
is a conical differential operator on   $M$ of order $m$ and only finitely many operators   $D_h$ are nonzero, so that the sum in  \eqref{eq-op12a} is actually finite.
\end{definition}

The aim of the present paper is to study ellipticity conditions, which guarantee   Fredholm property of operator~\eqref{eq-op12a}.

\section{Main results}

As usual in elliptic theory, the conditions, under which one obtains Fredholm property, are formulated in terms of the principal symbol of the operator. We shall refer to the principal symbol as the ``symbol'' for brevity. In our situation the symbol is a pair: the symbol on the main stratum (interior symbol) and the symbol at the conical point (conormal symbol\footnote{In the literature, one can also find the term ``indicial family''.}).
Let us give the definitions of these symbols.

\paragraph{Interior symbol.} To describe this symbol,  we need to extend the group action to the cotangent bundle of the manifold.

Denote the cotangent bundle of $M$ by $T^*M$ (e.g., see \cite{NSScS99}).

{ Recall the definition of $T^*M$. Let $\overline{M}=(M\setminus pt)\cup \Omega $ be the blowup of $M$, which is a manifold with boundary equal to $\Omega$. Denote by $r$ and $\omega$  the normal and tangent variables at the boundary.
By $TM$ we denote the vector bundle over $\overline{M}$, whose sections are generated as a  $C^\infty(\overline{M})$-module by vector fields 
$$
\frac\partial{\partial r}, \quad \frac 1 r \frac\partial{\partial \omega},
$$
which are of unit length in the conical metric of the form   $dr^2+r^2 d\omega^2$.

By  $T^*M\in\Vect(\overline{M})$ we denote the dual bundle (with respect to the natural pairing of vectors and covectors). One has the  vector bundle isomorphism
\begin{equation}\label{eq-isa3}
\begin{array}{ccc} T^*M|_{\partial \overline M} &\simeq &  T^*\Omega\times\mathbb{R},\\
    a dr+ b rd\omega &\mapsto & (b d\omega,a).
\end{array}
\end{equation}} 

A diffeomorphism $g:M\to M$ of the main manifold induces  diffeomorphism of the cotangent bundle  
$$
\partial g= ({}^tdg)^{-1}:T^*M \longrightarrow T^*M,
$$
which is called the {\em codifferential} of $g$. Here $dg:T_xM\to T_{g(x)}M$ stands for the differential, while ${}^tdg:T^*_{g(x)}M\to  T_x^*M$ stands for the dual mapping.

Let us define the symbol of operator  \eqref{eq-op12a} at a point $(x_0,\xi)\in T^*_0M$ of the cotangent bundle off the zero section using the method of frozen coefficients. Note that the operator under consideration is essentially nonlocal:  equation   $Du=f$ relates values of the unknown function   $u$ and its derivatives on the entire \emph{orbit} $Gx_0\subset M$ rather than  at a single point $x_0\in M$. For this reason, unlike the classical (nonsingular) case,   we  freeze the coefficients of the operator on the entire orbit of  $x_0$. Performing such procedure, and applying Fourier transform from $x$ to $\xi$ at each point of the orbit, we can define the symbol as a function on the cotangent bundle off the zero section, with values in operators acting on functions on the orbit. A direct computation (which we place in the appendix at the end of the paper) gives the following expression for the interior symbol   (cf. \cite{Sav12}):
\begin{equation}\label{traj-symbol1}
\sigma_0(D)(x_0,\xi)=\sum_{h\in G}  
\sigma(D_h) \Bigl(g^{-1}(x_0),\partial g^{-1}(\xi)\Bigr)\mathcal{T}_h:l^2(G,m_{s,\gamma})\longrightarrow
l^2(G,m_{s-m,\gamma+m}).
\end{equation}
Here
\begin{itemize}
\item the symbol acts on functions on the orbit   $Gx_0$, which is identified with $G$ using the mapping $g \mapsto g^{-1}(x_0)$;

\item $(\mathcal{T}_hw)(g)=w(g h)$ is the right shift of functions on the group;

\item the interior symbol of $D_h$ at a point $(x,\xi)$  is denoted by $\sigma_0(D_h)(x,\xi)$, while the expression $\sigma_0(D_h) (g^{-1}(x_0),\partial g^{-1}(\xi))$ is treated as an operator of multiplication of functions   on $G$:
$$
w(g) \longmapsto  \sigma_0(D_h) (g^{-1}(x_0),\partial g^{-1}(\xi)) w(g);
$$

\item the elements of the space $l^2(G,m_{s,\gamma})$ are  square summable  functions $ w(g) $ with respect to the weight function   on $G$:
\begin{equation}\label{measure-invariant1}
m_{s,\gamma}(g)=  \frac{{\partial g^{-1}}^{*} \mu_\gamma  }{\mu_\gamma} (x_0,\xi)\cdot |\partial g^{-1}\xi|^{2s}.
\end{equation}
Hereinafter, we omit the dependence of the weight function on  the variables $(x_0,\xi)$ for brevity.
\end{itemize}

\begin{definition}\label{def1} The operator \eqref{traj-symbol1} is the {\em  interior symbol} of a $G$-operator  $D$ at a point $(x_0,\xi)\in T^*_0M$.
\end{definition}

\begin{remark}
The expression \eqref{traj-symbol1} has the following natural meaning. Namely, according to this formula:  
\begin{itemize}
\item the symbol of the shift operator $T_h$ is a right shift  $\mathcal{T}_h$ on the group;
\item the symbol of differential operator $D_h$ is an operator of multiplication by the values of the symbol   $\sigma(D_h)$ at the corresponding points on the orbit;
\item finally, the symbol of  $D=\sum_h D_hT_h$ is the sum of products of the symbols of $D_h$ and $T_h$.
\end{itemize}
\end{remark}

\begin{remark} 
In the general case, interior symbol depends on   $x_0,\xi$ in a quite complicated way. For example, it is frequently discontinuous in   $x_0,\xi$. This is related with the fact that the structure of orbit can depend on an initial point on the orbit in a very complicated way.  
\end{remark}

\begin{proposition}\label{prop-4}
Given two $G$-operators $D_1$ and $D_2$, we have
$$
\sigma_0(D_1D_2)=\sigma_0(D_1)\sigma_0(D_2).
$$
\end{proposition}
\begin{proof}
This follows directly from the fact that the symbol was obtained by the method of frozen coefficients. However, the desired formula can also be obtained directly using    \eqref{traj-symbol1} only.
\end{proof}

Let us compute the restriction  $\sigma_0(D)|_{\partial T^*M}$ of the interior symbol to the boundary of the cotangent bundle, i.e., for  $r=0$. Recall that near the conical point we have coordinates $r,\omega$ and dual coordinates $\tau,\eta$.
For the operator \eqref{eq-op12a}  and $(x,\xi)=(0,\omega ,\tau,\eta)\in \partial T^*M$ we have
\begin{equation}\label{traj-symbol2}
\sigma_0(D)(0,\omega ,\tau,\eta)=\sum_{h\in G}  
\sigma_0(D_h) \Bigl(0,g^{-1} (\omega),\tau, \partial g^{-1}(\eta)\Bigr)\mathcal{T}_h:l^2(G,m_{s,\gamma})\longrightarrow
l^2(G,m_{s-m,\gamma+m}),
\end{equation}
where the weight function $m_{s,\gamma}$ is equal to
\begin{equation}\label{measure-invariant2}
m_{s,\gamma}(g)= e^{-2\beta_g\gamma} \cdot \frac{{\partial g^{-1} }^{*} \mu  }{\mu} (\omega)\cdot \Bigl(\tau^2+|\partial g^{-1} \eta|^2\Bigr)^{s},\quad \text{here } \mu \text{ is a measure on }\Omega.
\end{equation}
In this formula  we used the special form \eqref{eq-diff1} of $g$ near the conical point.
In particular, it follows from \eqref{measure-invariant2} that
$$
m_{s,\gamma}(g)= e^{-2\beta_g\gamma} 
$$
 if $g:\Omega\to\Omega$ is an isometry.

\paragraph{Conormal symbol.}
To obtain the symbol of the operator \eqref{eq-op12a} at the conical point, we note that by assumption this point is fixed under the group action on the manifold. Hence, to obtain the symbol it suffices to freeze the coefficients at this point, i.e., to set   $r=0$ in the coefficients. We get
$$
\sum_{h\in G} r^{-m}D_h\left(0,-r\frac{\partial}{\partial r},\omega,-i\frac{\partial}{\partial \omega}\right) T_h:H^{s,\gamma}(\mathbb{R}_+\times\Omega)\longrightarrow H^{s-m,\gamma+m}(\mathbb{R}_+\times\Omega).
$$
Now we omit the factor $r^{-m}$ and apply Mellin transform $\mathcal{M}_{r\to p}$ (see \eqref{eq-mel2}).
As a result in the coordinates $\omega,p$ we obtain the operator family 
\begin{equation}\label{eq-op12}
  \sum_{h\in G} D_h\left(0,p,\omega,-i\frac{\partial}{\partial \omega}\right) T'_h e^{\beta_h p}:H^s(\Omega)\longrightarrow H^{s-m}(\Omega).
\end{equation} 
where  $p\in L_\gamma$,  while 
$$
T'_h: H^s(\Omega) \longrightarrow H^s(\Omega), \qquad (T'_h u) (\omega)=u(h^{-1}(\omega))
$$
is the shift operator on the base of the cone. Note that to obtain  \eqref{eq-op12}, we used the fact that the Mellin transform takes operator $-r\partial/\partial r$ to the operator of multiplication by  $p$,  and the dilation operator $u(r)\mapsto u(re^{-\beta})$  to the operator of multiplication by $e^{\beta p}$, while the elements of the weighted space are taken to functions on the line $L_\gamma$, ranging in the space $H^s(\Omega)$.

\begin{definition}
The operator family \eqref{eq-op12} is the  {\em conormal symbol of $G$-operator $D$} and is denoted by $\sigma_c(D)$.
\end{definition}
 
\begin{proposition}Let  
$$
D_1:H^{s,\gamma}(M)\longrightarrow H^{s-m_1,\gamma+m_1}(M),\qquad   D_2:H^{s-m_1,\gamma+m_1}(M)\longrightarrow H^{s-m_1-m_2,\gamma+m_1+m_2}(M)   
$$
be two $G$-operators of orders  $m_1,m_2$, respectively. Then their composition is a  $G$-operator and one has
$$
\sigma_c(D_2D_1)(p)=\sigma_c(D_2)(p+m_1)\sigma_c(D_1)(p).
$$
\end{proposition}
This proposition is proved by a direct computation. We note only that the shift of the argument of the conormal symbol in the latter formula is due to the presence of the factor   $r^{-m}$ in \eqref{eq-op12a}  (cf. \cite{NSScS99} in the classical case).

\paragraph{Ellipticity and Fredholm property.}

\begin{definition}\label{def-ell4}
A $G$-operator $D$ (see \eqref{eq-op12a}) on a manifold with conical points is  {\em elliptic}, if it has invertible
\begin{enumerate}
\item[1)] interior symbol $\sigma_0(D)(x,\xi)$ for all $(x,\xi)\in T^*_0M$; 
\item[2)] conormal symbol $\sigma_c(D)(p)$ for all  $p$  on the weight line $L_\gamma=\{\re p=\gamma\}$.
\end{enumerate}
\end{definition}

\begin{theorem}\label{th-finite1}
If a $G$-operator
$$
 D=\sum_{h\in G} D_h T_h:H^{s,\gamma}(M)\longrightarrow H^{s-m,\gamma+m}(M)
$$
is elliptic, then it is Fredholm.
\end{theorem}
 
The proof of Theorem~\ref{th-finite1} uses the theory of $C^*$-algebras and is given in Section~\ref{sec-c1}.

\begin{remark}
The conormal symbol  $\sigma_c(D)(p)$ is a family of $G$-operators on $\Omega$, depending on the parameter $p$. In Section \ref{sec-c1} we develop elliptic theory for such families and show that the family   $\sigma_c(D)(p)$ is invertible for  $p\in L_\gamma$ large, provided the interior symbol   $\sigma_0(D)$ is elliptic for the weight $\gamma$.
\end{remark}

\begin{remark}
The results of this section extend to the case of pseudodifferential operators ($\psi$DO), i.e., the case, when the operators   $D_g$,  which appear in the definition of a $G$-operator, are pseudodifferential operators on $M$.  Concerning the theory of $\psi$DO on manifolds with conical points see, for example, the monograph  \cite{NSScS99} and the references  there.
\end{remark}

\section{Example: $\mathbb{Z}$-operators on the sphere}

\paragraph{Sphere as a $\mathbb{Z}$-manifold.}
Consider the sphere $\mathbb{S}^2$ as the extended complex plane with complex coordinate  $z$ and treat the points of the sphere $z=0$ and  $z=\infty$ as conical points  denoted by $S$ and $N$. On the sphere, we consider the action of group $\mathbb{Z}$, defined by iterations of the diffeomorphism 
$$
g(z)=zz_0, \text{ with } z_0=e^{ \beta+i\varphi_0} \text{ fixed.}
$$
(see Figure~\ref{fig1q}.) For definiteness, we assume that  $\beta>0$ and  $\varphi_0\in[0,2\pi)$.

\begin{figure}
 \begin{center}
  \includegraphics[width=5cm]{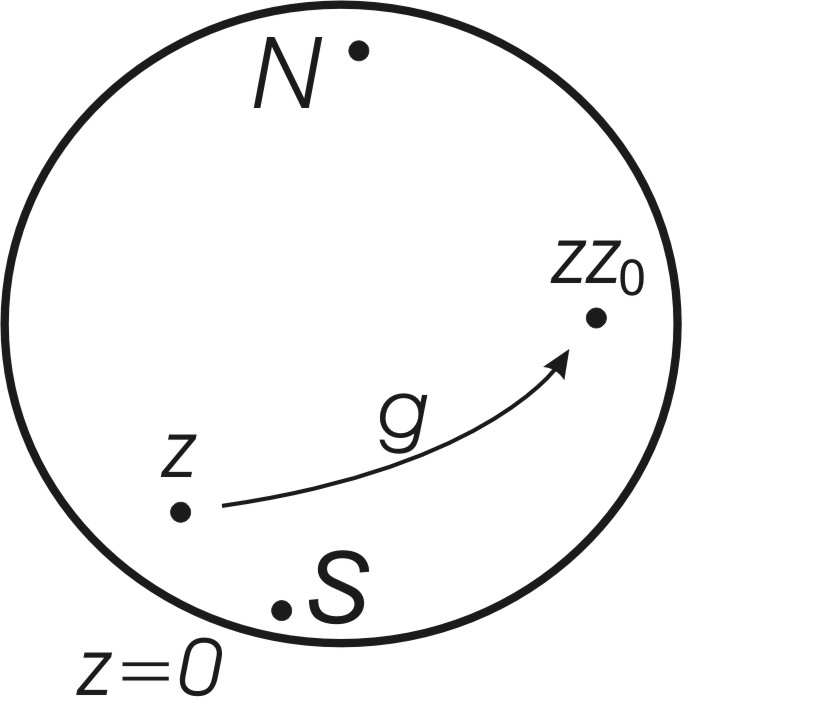} \end{center}\caption{Diffeomorphism of $\mathbb{S}^2$.}\label{fig1q}
\end{figure}

We replace  the complex coordinate $z$ by  real coordinates  $t,\varphi$ using the change of variables
$z=e^{-t+i\varphi}$, which defines the diffeomorphism  (blowup)
\begin{equation}\label{eq-blowup1}
\mathbb{S}^2\setminus\{N,S\}\simeq \mathbb{S}^1_\varphi\times \mathbb{R}_t
\end{equation}
of the complement of the conical points and the infinite cylinder. Moreover, under this diffeomorphism neighborhoods of the conical points  $S,N$  correspond to neighborhoods of the sections at $t=\pm\infty$, while the diffeomorphism  $g$ is equal to:
$$
g(\varphi,t)=(\varphi+\varphi_0,t-\beta),
$$
i.e., it is just a helical motion of the cylinder \eqref{eq-blowup1}.

\paragraph{Symbols of  $G$-operators.}

On $\mathbb{S}^2$ we consider  conical $\mathbb{Z}$-operators
$$
D=\sum_{0\le k\le m,\; l\in\mathbb{Z}} D_{l,k}\left(-i\frac{\partial}{\partial \varphi}\right)\left(-r\frac{\partial}{\partial r}\right)^k T^l:H^{s,\gamma }(\mathbb{S}^2)\longrightarrow H^{s-m,\gamma }(\mathbb{S}^2)
$$
with constant coefficients in  $\varphi,r$ (here we omitted the factor $r^{-m}$ right away), while the shift operator is equal to 
$$
(Tu)(\varphi,r)=u(\varphi-\varphi_0, r e^{-\beta}).
$$ 
Now $D$ acts in spaces with the weight $\gamma=(\gamma_+,\gamma_-)$, where the weight $\gamma_+$ is at $S$ and the weight  $\gamma_-$ is at  $N$ (one can check that  $D$ is a conical operator at the conical point $r=\infty$ using the inversion $r'=1/r$).

Let us compute the   interior and the conormal symbols of  $D$.

The interior symbol is equal to 
\begin{equation}\label{eq-ints1}
\sigma_0(D)(\eta,\tau)=
\sum_{0\le k\le m,\; l\in\mathbb{Z}} D_{l,k}\left(\eta\right)\left(i\tau \right)^k \mathcal{T}^l:l^2(\mathbb{Z},m_{s,\gamma})\longrightarrow l^2(\mathbb{Z},m_{s,\gamma}),
\end{equation}
where  $\eta^2+\tau^2\ne 0$ and  $(\mathcal{T}w)(n)=w(n+1)$.

The conormal symbols at  $r=0$ and  $r=\infty$ are equal to 
$$
\sigma_{c}(D)(p)=
\sum_{0\le k\le m,\; l\in\mathbb{Z}} D_{l,k}\left(-i\frac{d}{d \varphi}\right)p^k {T'}^l e^{-\beta l p}:H^s(\mathbb{S}^1)\longrightarrow H^{s-m}(\mathbb{S}^1),\qquad \re p=\gamma_\pm.
$$

By Theorem~\ref{th-finite1}   operator $D$ is Fredholm if its interior symbol  $\sigma_0(D)(\eta,\tau)$ is invertible whenever $\eta^2+\tau^2\ne 0$ and the conormal symbol  $\sigma_{c }(D)(p)$ is invertible on the weight lines  $\re p=\gamma_\pm$. To verify  these invertibility conditions, let us compute the weight function  $m_{s,\gamma}$ in \eqref{eq-ints1}.

\paragraph{Computation of the weight function.} A computation shows that the measure  $\mu_\gamma$ is equal to 
$$
\mu_\gamma=\left\{
  \begin{array}{ll}
   \displaystyle   r^{ 2\gamma_+}d\varphi \frac {dr}r, & \text{ if } r<1,\vspace{2mm}\\
    \displaystyle   r^{ 2\gamma_-}d\varphi \frac {dr}r, & \text{ if } r>2. 
  \end{array}
\right.
$$
We fix the metric   $d\varphi^2+r^{-2}dr^2$ on the cylinder.

\begin{proposition}
The weight function
\begin{equation}\label{measure-invariant1a}
 m_{s,\gamma}(n)= \left(\frac{{\partial g^{-n}}^{*}(\mu_\gamma (\eta^2+\tau^2)^{s})}{\mu_\gamma}\right)(t_0,\varphi_0,\eta,\tau).
\end{equation}
(see~\eqref{measure-invariant1}) up to equivalence does not depend on $s$ and is equal to 
\begin{enumerate}
 \item[1)] at a point in the interior of  $T^*_0M$:
 \begin{equation}\label{eq-53}
    m_{s,\gamma}(n)\simeq
      \left\{
  \begin{array}{ll}
    e^{ 2\gamma_-\beta n}, & \text{ if } n\ge 0,\\
    e^{ 2\gamma_+\beta n}, & \text{ if } n<0; 
  \end{array}
\right.
 \end{equation}
 \item[2)] at a point on the boundary $\partial T^*_0M=\{r=0 \text{ and } r=\infty\}$:
 $$
    m_{s,\gamma}(n)\simeq
      \left\{
  \begin{array}{ll}
    e^{ 2\gamma_+\beta n}, & \text{ if } r=0,\\
    e^{ 2\gamma_-\beta n}, & \text{ if } r= \infty. 
  \end{array}
\right.
 $$
\end{enumerate}
\end{proposition}
This proposition is proved by a direct computation using  \eqref{measure-invariant1a}.

\paragraph{Examples.}

Consider the operator
\begin{equation}\label{eq-15a}
D=-r\frac{\partial }{\partial r} +(a+bT)\left(-i\frac{\partial }{\partial \varphi}\right)+(c+dT): H^{s,\gamma }(\mathbb{S}^2)
\longrightarrow H^{s-1,\gamma }(\mathbb{S}^2).
\end{equation}

1. The interior symbol of this operator is equal to 
\begin{equation}\label{eq-22}
\sigma_0(D)(\eta,\tau)=
i\tau +(a+b\mathcal{T})\eta:l^2(\mathbb{Z},m_{s,\gamma})\longrightarrow l^2(\mathbb{Z},m_{s-1,\gamma}),
\qquad (\mathcal{T}w)(n)=w(n+1),
\end{equation}
where the weight  $m_{s,\gamma}$ was defined in  \eqref{eq-53}.
It is easy to algebrize operator \eqref{eq-22}. Let us formulate the corresponding result.
\begin{lemma}[cf. \cite{Ster7}]\label{lemma-3}
Given $\gamma_+>\gamma_-$,  the Fourier transform 
\begin{equation}
  \begin{array}{ccc}
    \mathcal{F}: l^2(\mathbb{Z},m_{s,\gamma}) & \longrightarrow & \mathcal{O}(K_\gamma)\vspace{2mm}\\
     \{w(n)\} & \longmapsto & \sum_n w(n) \zeta^n  
  \end{array}
\end{equation}
takes the space  $l^2$ with the weight $m_{s,\gamma}$ to the space of holomorphic functions in the annulus
\begin{equation}\label{eq-ring1}
K_\gamma=\{r<|\zeta|<R\},\qquad r=e^{-\gamma_+\beta},\quad R=e^{-\gamma_-\beta},
\end{equation}
which are square integrable on  the boundary. If  $\gamma_->\gamma_+$, then the range of  Fourier transform is the space dual to the space of holomorphic functions in the annulus with  outer radius    $r$ and   inner radius $R$. 
\end{lemma}
\begin{proof}
The proof is based on the fact that  $w\in l^2(\mathbb{Z},m_{s,\gamma})$ if and only if 
$$
\sum_{n\ge 0} |w(n)|^2 e^{ 2\gamma_-\beta n}<\infty,\qquad 
\sum_{n< 0} |w(n)|^2 e^{ 2\gamma_+\beta n}<\infty.
$$
These conditions mean that the Fourier transform of this function, i.e., the Laurent  series
$$
\sum_n w(n)\zeta^n
$$
is absolutely convergent in the mentioned annulus. 
\end{proof}
By Lemma~\ref{lemma-3}, we transform the symbol to the operator of multiplication 
\begin{equation}\label{eq-24}
\sigma_0(D)(\eta,\tau)=
i\tau +(a+b\zeta^{-1})\eta:\mathcal{O}(K_\gamma)\longrightarrow \mathcal{O}(K_\gamma)
\end{equation}
acting in the space of analytic functions in the annulus.  Hence, the interior symbol is invertible if and only if this function is invertible in the annulus.
\begin{lemma}
The interior symbol \eqref{eq-22} is invertible if and only if
\begin{equation}\label{eq-20}
|\re a|> \frac{|b|}r,
\end{equation}
where  $r=\min(e^{-\gamma_+\beta},e^{-\gamma_-\beta})$.
\end{lemma}
\begin{proof} 
The interior symbol is invertible for all  $\eta,\tau\ne 0$ if and only if
$$
 \re(a+b\zeta^{-1})\ne 0\text{ for all }\zeta\in K_\gamma.
$$ 
The numbers $a+b\zeta^{-1}$ run over the annulus with the outer radius $|b|/r$  with center at point $a$ (see Figure~\ref{fig1t}).
\begin{figure}
 \begin{center}
  \includegraphics[width=7cm]{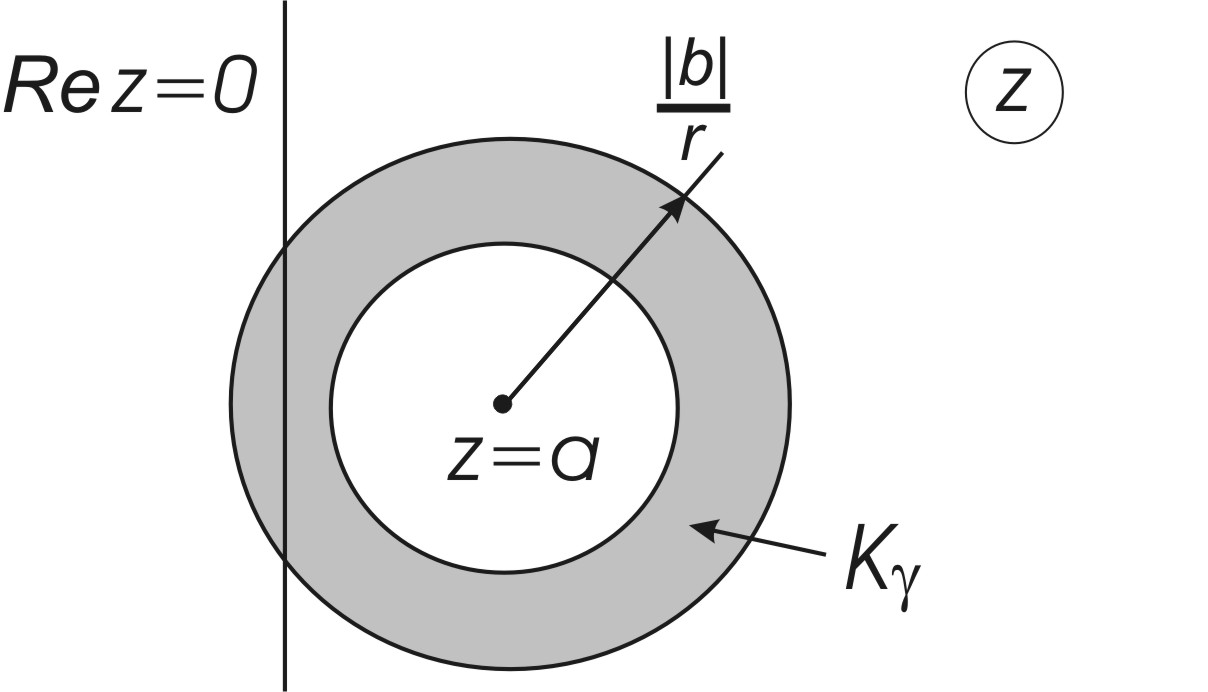} \end{center}\caption{The annulus $K_\gamma$ and the line $\re z=0$}\label{fig1t}
\end{figure}
The condition that the intersection of this annulus and the line  $\re z=0$ is empty is precisely  \eqref{eq-20}.
\end{proof}

We finally get the invertibility condition for the interior symbol in the form 
$$
\gamma_+,\gamma_-\le \frac{\ln|\re a|-\ln |b|}\beta.
$$

2. The conormal symbol of the operator \eqref{eq-15a} is
$$
\sigma_c(D)(p)=p+(a+b e^{-\beta p}T')\left(-i\frac{d }{d \varphi}\right)+(c+de^{-\beta p}T'):H^s(\mathbb{S}^1)\longrightarrow H^{s-1}(\mathbb{S}^1),
$$
where $(T'u)(\varphi)=u(\varphi-\varphi_0)$ is the shift operator on the circle. 
For $p\in L_\gamma$ large the conormal symbol is invertible  (under the assumption that the operator is interior elliptic).

Expanding functions on the circle in Fourier series, we obtain the ellipticity condition for the conormal symbol in the form: none of the equations 
\begin{equation}\label{eq-25}
p+(a+b e^{-\beta p-in\varphi_0})n+(c+de^{-\beta p-in\varphi_0})=0\quad \text{for all }n\in \mathbb{Z}
\end{equation}
has a solution  $p$ such that  $\re p=\gamma_\pm$.

Solving transcendental equations \eqref{eq-25}, we can obtain the singular points of the conormal symbol, that is the values of the weight for which the problem is not Fredholm. 

\begin{remark}
Note that the ellipticity condition can be effectively verified numerically: since Eq.~\eqref{eq-25} has no solutions in the vertical stripe  $|\re p|<N$  provided  $|p|^2+n^2$ is large,  all the solutions in the mentioned stripe are associated with a finite number of   $n$'s and a bounded domain of  $p$'s.
\end{remark}

\section{Example: $\mathbb{Z}^2$-operators on half-line}

Let us treat the half-line $\mathbb{R}_+$ as a manifold with two conical points and consider the operator
\begin{equation}\label{eq-aa5}
A =1+aT_\alpha+b T_\beta: H^{s,\gamma}(\mathbb{R}_+)\longrightarrow H^{s,\gamma}(\mathbb{R}_+), \qquad (T_\alpha u)(r)=u(re^{\alpha}),
(T_\beta u)(r)=u(re^{\beta}). 
\end{equation}
Here we choose weights $\gamma=(\gamma_+,\gamma_-)$  at $r=0$ and  $r=\infty$, correspondingly, and  suppose that the numbers  $\alpha,\beta$ are {\em incommensurable}. Let us     assume for definiteness that $\alpha,\beta>0$, $\gamma_+>\gamma_-$.

Let us obtain conditions under which the operator \eqref{eq-aa5} is Fredholm. For simplicity we set $s=0$ and apply Mellin transform \eqref{eq-mel2}. We obtain the operator
\begin{equation}\label{eq-c3}
\sigma_c(A)(p)=1+ae^{p\alpha}+be^{p\beta}:\mathcal{O}(S)\longrightarrow \mathcal{O}(S),
\end{equation}
where  $S=\{\gamma_-<\re p< \gamma_+\}$ is the vertical stripe in the complex plane, while  $\mathcal{O}(S)$ is the Mellin transform of the space $L^{2,\gamma}(\mathbb{R}_+)$. The elements of this space admit  holomorphic  continuation inside  this stripe (see  \cite{Ster7}).

One can show that the operator \eqref{eq-c3} is invertible in the space  $\mathcal{O}(S)$ if and only if the function
\begin{equation}\label{eq-61}
1+ae^{\gamma\alpha}e^{i\varphi}+be^{\gamma \beta} e^{i\psi}
\end{equation}
is nonzero for all $\gamma\in[\gamma_-,\gamma_+],$ $\varphi,\psi\in[0,2\pi]$. It is easy to write out the invertibility condition for the latter function explicitly:
given some  $\gamma$, the sum of the first two terms runs over the circle in the complex $z$-plane with center at  $z=1$ and radius  $|a|e^{\gamma\alpha}$, while the third summand runs over the circle with center at  $z=0$ and radius $|b|e^{\gamma\beta}$. Clearly, the function  \eqref{eq-61} is invertible if and only if  the mentioned circles do not intersect each other. Further, the intersection of two circles is nonempty if and only if the radii of the circles and the distance  between their centers satisfy   triangle inequalities  (see Figure~\ref{fig2a}):
\begin{equation}\label{eq-62}
 |a|e^{\alpha\gamma}\le 1+|b|e^{\beta\gamma},\quad |b|e^{\beta\gamma}\le 1+|a|e^{\alpha\gamma},\quad 1\le |a|e^{\alpha\gamma}+|b|e^{\beta\gamma}.
\end{equation}
\begin{figure}
 \begin{center}
  \includegraphics[width=6cm
 ]{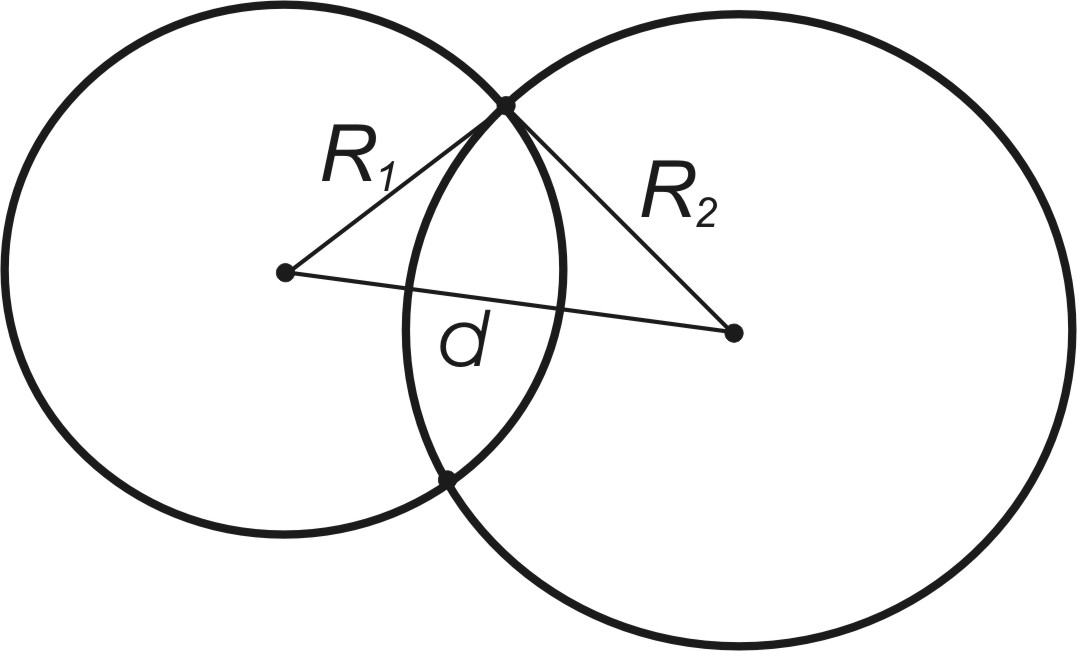} \end{center}\caption{Intersection of two circles.}\label{fig2a}
\end{figure}
More precisely, the intersection of two circles is nonempty if and only if all  three inequalities are valid. So, our operator is elliptic if for all $\gamma\in[\gamma_-,\gamma_+]$ one of the mentioned inequalities is not valid. For fixed  $\gamma$ the domain, in which one of inequalities is not valid, is shown  in Figure~\ref{fig2b} (the domains, in which inequality one, two, three is not valid is denoted by I, II, III, correspondingly). Further, the intersection of all such domains over all  $\gamma\in[\gamma_-,\gamma_+]$ is shown in Figure~\ref{fig3a} (we omit the corresponding elementary computations). 
\begin{figure}
 \begin{center}
  \includegraphics[width=13cm
 ]{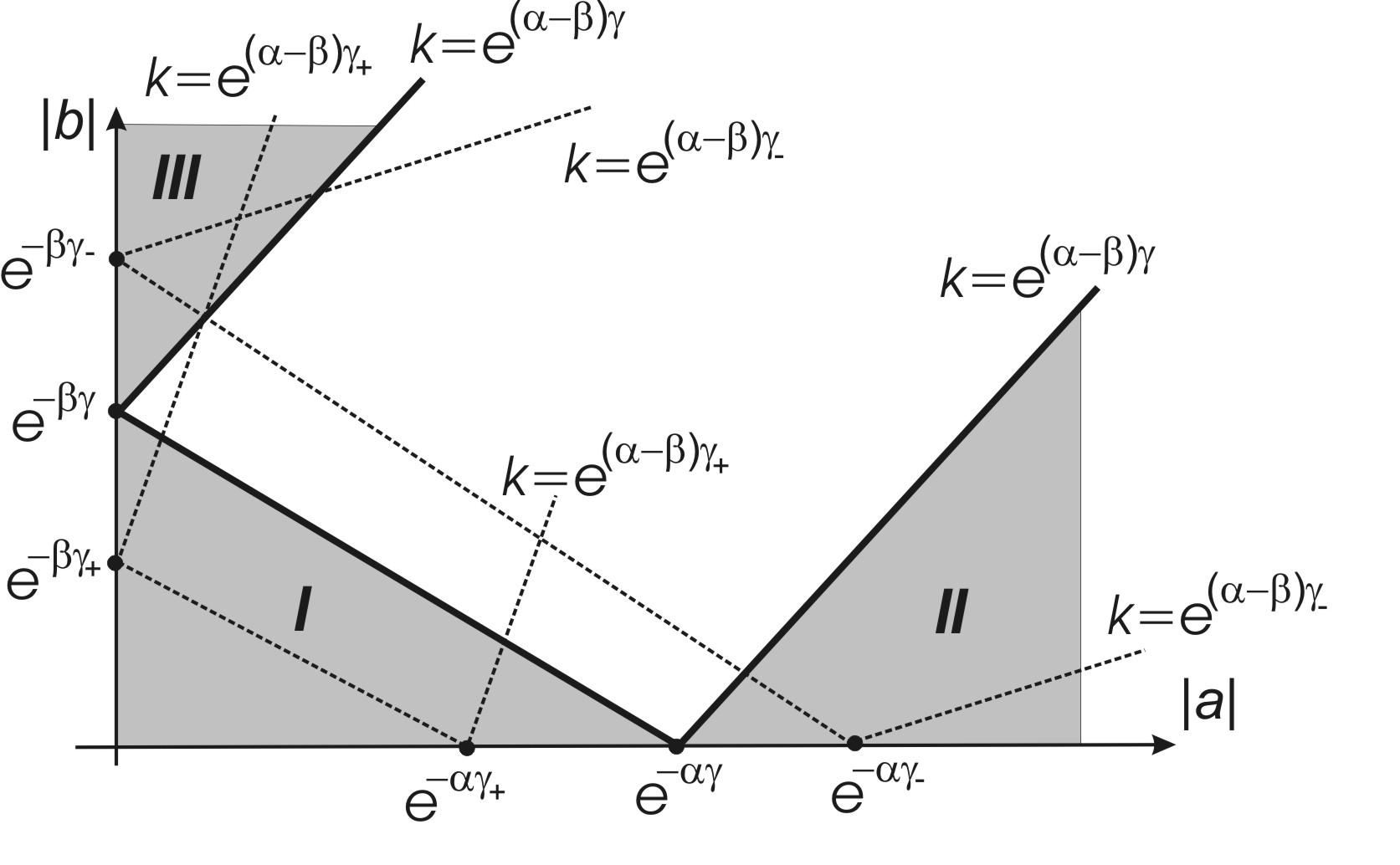} \end{center}\caption{Values of $a,b$, for which circles $|z|=|b|e^{\beta\gamma}$ and $|z-1|=|a|e^{\alpha\gamma}$ do not intersect.}\label{fig2b}
\end{figure}

We see that the domain of parameters $a,b$, for which the operator is invertible, consists of three components. The meaning of this decomposition into three components is simple: the operator  $A$ is equal to the sum of three terms and this operator is invertible if and only if one of the terms is dominating  (the first inequality corresponds to the domination of the identity operator, the second corresponds to  $aT_\alpha$, the third corresponds to $bT_\beta$). 
 
\begin{figure}
 \begin{center}
  \includegraphics[width=13cm
 ]{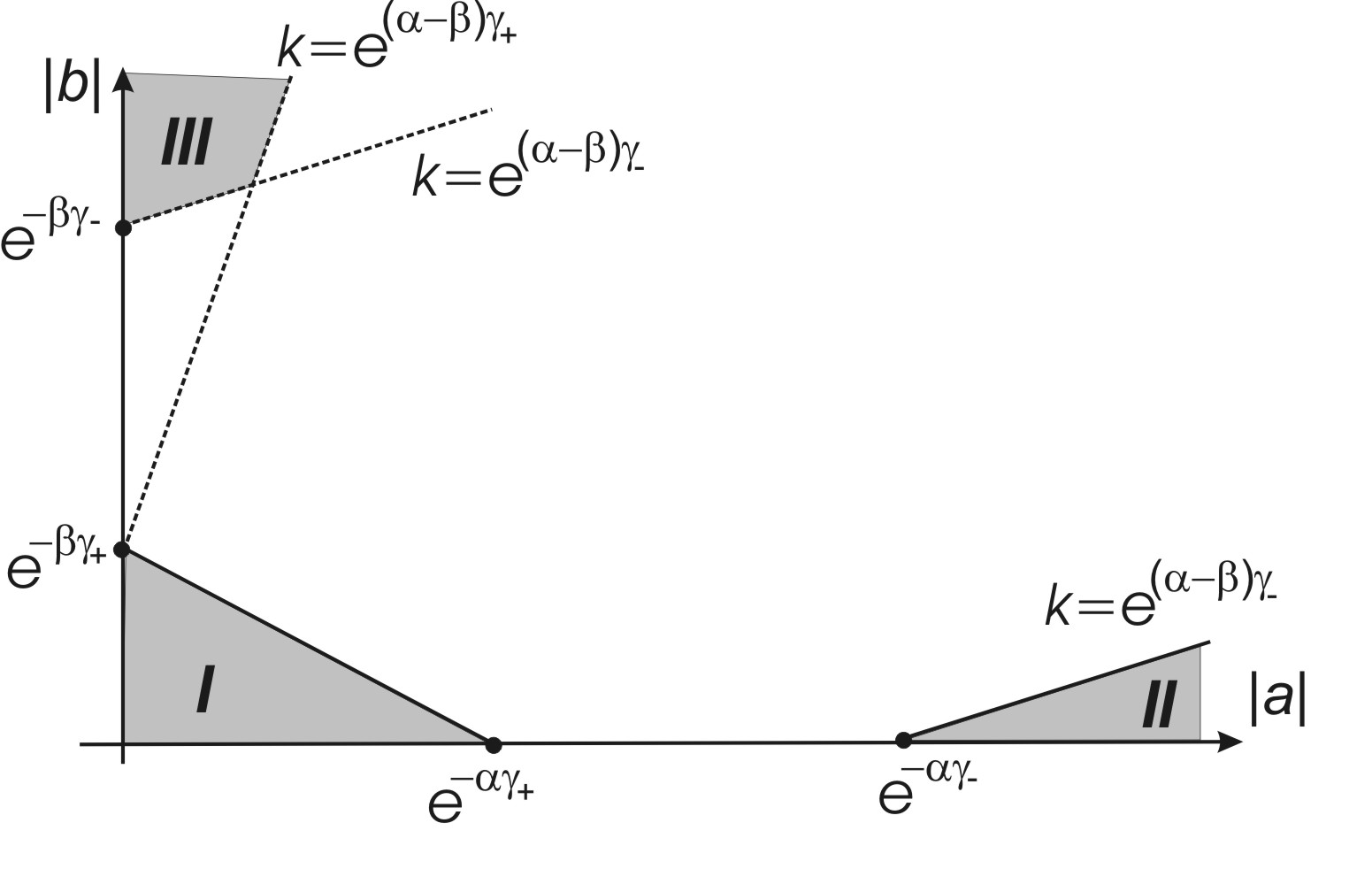} \end{center}\caption{Domain, in which the operator \eqref{eq-c3} is invertible.}\label{fig3a}
\end{figure}

\section{Proofs of the main results ($C^*$-theory)}\label{sec-c1}

The aim of this section is to prove Theorem~\ref{th-finite1}, i.e., to obtain Fredholm property of elliptic  $G$-operators.  As usual, the Fredholm property is proved by constructing an almost-inverse operator, i.e., inverse up to compact operators. Clearly,  given   $G$-operator   \eqref{eq-op12a} (where the sum is actually finite), a similar sum for an almost-inverse operator is (as a rule) infinite and for this reason one has to work with $G$-operators equal to infinite sums. Unfortunately, a direct consideration of such sums in the general situation, which we consider, looks quite difficult. However, there is a general approach, which enables one to overcome difficulties, related with infinite series. This approach is based on systematic use of the language of  $C^*$-algebras and their crossed products (see \cite{AnLe1}; below we follow the exposition in \cite{SaSt33}).

\paragraph{$C^*$-algebra of $\psi$DO $\Psi(M)$.}

Let us briefly recall the structure of the  $C^*$-algebra
$$
\Psi(M)\subset \mathcal{B}L^{2}(M),
$$
of  zero order $\psi$DOs on a manifold $M$ with conical singularities acting on the space $L^{2}(M)\equiv H^{0,0}(M)$, that is  $s=\gamma=0$. 

This algebra can be  obtained (e.g., see \cite{NaSaSt5}) as the norm closure of zero order operators with smooth interior symbols and smooth conormal symbols.  This algebra is endowed  with  the symbol mapping 
$$
\sigma=(\sigma_0,\sigma_c): \Psi(M) \longrightarrow C(S^*M)\oplus \Psi_p(\Omega ) 
$$
which takes a $\psi$DO $D$ to its interior symbol  $\sigma_0(D)$ (continuous function on the cosphere bundle $S^*M=T^*_0M/\mathbb{R}_+$) and conormal symbol $\sigma_c(D)$ (family of $\psi$DOs on the base of the cone $\Omega$, continuously depending on parameter $p\in L=\{\re p=0\}$). Here the interior symbol and the conormal symbol are subject to compatibility condition: the restriction of the interior symbol to the boundary  $\partial S^*M$ is equal to the symbol of the conormal symbol, which is considered as a parameter-dependent family. Thus elements of the algebra 
$\Psi(M)$ can be treated as $\psi$DOs on $M$ with continuous symbols.

\paragraph{$G$-operators and crossed products.} 

$G$-action on $M$ induces the unitary representation of $G$ on $L^{2 }$:
$$
\widetilde{T}_g:L^{2 }(M)\longrightarrow L^{2 }(M),\quad g\in G.
$$
$$
\widetilde{T}_g= \mu ^{-1/2} T_g \mu ^{1/2}\equiv 
    \left(\frac{{g^{-1}}^*\mu }{\mu }\right)^{1/2} T_g,
$$
where $\mu=\mu_0$ is the measure, which is used in the definition of the space  $L^2(M)$.

The group  $G$ acts on $\Psi(M)$ on the left as 
$$
A\in \Psi(M), g\in G \longmapsto  \widetilde{T}_g A \widetilde{T}_g^{-1}\in \Psi(M)
$$
and one defines the maximal  $C^*$-crossed product  $\Psi(M)\rtimes G$ (concerning the theory of crossed products see, e.g.,  \cite{Ped1}). Elements of the crossed product are families  $\{D_g\}$ of $\psi$DOs on $M$, parametrized by the group. A family is called finite, if it has only finitely many nonzero elements. The set of finite families is denoted by  $\Psi(M)\rtimes_{alg} G$ and called the {\em algebraic} crossed product. We shall use the following universal property of the crossed product  $A\rtimes G$ of algebra  $A$ and group $G$, acting on  $A$ by automorphisms: an arbitrary covariant representation, i.e., a pair of representations 
$$
\rho_0:A\to \mathcal{B}H, \rho_1:G\to \mathcal{B}H \text{ such that }\rho_0(g(a))=\rho_1(g)\rho_0(a)\rho_1(g)^{-1},
$$ 
(the representation $\rho_1$ is assumed to be unitary) induces a homomorphism of the crossed product:
$$
\rho: A\rtimes G \longrightarrow \mathcal{B}H,
$$
which on the elements of the algebraic crossed product has the form 
$$
\rho\{a_h\}=\sum_{h\in G }  \rho_0(a_h)\rho_1(h).
$$
We now define a mapping, which takes a finite family  $\{D_h\}$ to the $G$-operator
$$
 \sum_h D_h \widetilde{T}_h: L^{2 }(M) \longrightarrow L^{2 }(M).
$$ 
By the universal property this mapping extends to the $C^*$-homomorphism
$$
\begin{array}{ccc}
  \pi: \Psi(M)\rtimes G & \longrightarrow & \mathcal{B}  / \mathcal{K} \vspace{2mm}\\
    \{D_h\} & \longmapsto  &\sum_h D_h \widetilde{T}_h.
\end{array}
$$
Here  $\mathcal{B}$ ($\mathcal{K}$) stands for the algebra of bounded (compact) operators acting on  $L^{2 }(M)$. 

Our aim in this section is to study Fredholm property for  $G$-operators. Obviously, the Fredholm property of a $G$-operator is equivalent to the invertibility of the corresponding element in the Calkin algebra  $\mathcal{B}/\mathcal{K}$, i.e.,  invertibility of the element  $\pi(a)$ for $a\in \Psi(M)\rtimes G$. To study the invertibility problem, we introduce the notion of symbol for $G$-operators and the corresponding quantization. 

Since the conormal symbol of a $G$-operator is a family of  $G$-operators with parameter (see~\eqref{eq-op12}), let us study this notion.

\paragraph{$G$-operators with parameter.} 

$G$ acts on the cosphere bundle  $S^*M=T^*_0M/\mathbb{R}_+$ and the boundary of this bundle, denoted by $\partial S^*M$, is invariant with respect to this action. Hence, we can define the crossed product  $C(\partial S^*M)\rtimes G$. 
Since the conical point has a punctured  neighborhood diffeomorphic to the product $\Omega\times \mathbb{R}_+$, one has a diffeomorphism $\partial S^*M\simeq S(T^*\Omega\times\mathbb{R}_p)$.

Denote by 
$$
\Psi_p(\Omega)\subset C(L ,\mathcal{B}L^2(\Omega)),\qquad\text{ where } L =\{\re p=0\}, 
$$  
the norm closure of the algebra of classical zero-order $\psi$DOs  on $\Omega$ with parameter $p\in L $ (e.g., see  \cite{NaSaSt5}).  Here  $C(L ,\mathcal{B}L^2(\Omega))$ stands for the algebra of continous bounded functions on the weight line  $L $ ranging in bounded operators on  $L^2(\Omega)$. 
\begin{definition}
{\em $G$-symbol with parameter} $p$ is an element
$$
 a\in C(S(T^*\Omega\times \mathbb{R}_p))\rtimes G.
$$ 
\end{definition}
Let us define  quantization  of $G$-symbols with parameter by the formula
\begin{equation}\label{eq-opp1}
Op(a)= \sum _h Op(a_h)(p)\widetilde{T}'_h e^{p\beta_h}:L^2(\Omega)\longrightarrow L^2(\Omega),\text{ where }a=\{a_g\},\qquad p\in L ,
\end{equation}
and $\widetilde{T}'_h$ is the  unitary representation of $G$ on $L^2(\Omega)$ by shift operators on $\Omega$.
The formula \eqref{eq-opp1} defines  a quantization mapping on the algebraic crossed product, which extends by the universal property to the  $C^*$-crossed product as a  $C^*$-homomorphism denoted by
$$
Op: C(S(T^*M\times \mathbb{R}_p))\rtimes G\longrightarrow C\bigl(L ,\mathcal{B}\bigr)/C_0\bigl(L ,\mathcal{K}\bigr).
$$
Here  $C_0\bigl(L ,\mathcal{K}\bigr)$ is the ideal of continuous compact-valued functions vanishing at infinity. 

\begin{definition}
A {\em parameter-dependent family of $G$-operators} is a family of operators
$$
 A(p)\in C\bigl(L ,\mathcal{B}\bigr),
$$  
such that
$A(p)-Op(a)\in C_0(L ,\mathcal{B}/\mathcal{K})$
for some symbol $a\in C(S(T^*M\times \mathbb{R}_p))\rtimes G$.
\end{definition}
An element $a$ will be called the  {\em symbol of parameter-dependent family} and denoted by $\sigma(A(p))$.

\begin{definition}
The algebra of  {\em parameter dependent families of $G$-operators }   is defined as
\begin{equation}\label{eq-param1}
\Psi^G_p(\Omega)= \Bigl\{(a, A(p))\in C(\partial S^*M)\rtimes G\oplus C(L ,\mathcal{B}L^2(\Omega)) \;|\; 
Op(a)-A(p)\in C_0(L ,\mathcal{K})\Bigr\}.
\end{equation}
\end{definition}

By construction we have the short exact sequence of  $C^*$-algebras
\begin{equation}\label{eq-phi3}
\begin{array}{rccccl}
0\longrightarrow C_0(L ,\mathcal{K}) & \longrightarrow & \Psi^G_p(\Omega) & \longrightarrow & C(\partial S^*M)\rtimes G & \longrightarrow 0 \vspace{1mm}\\
A(p) &  \longmapsto &  (0,A(p)) \vspace{1mm}\\
 & & (a,A(p)) & \longmapsto & a
\end{array}
\end{equation}
Since the sequence  \eqref{eq-phi3} is exact, we obtain   the following corollary.
\begin{corollary}\label{cor-7}
If a parameter-dependent family  $A(p)$  is elliptic, i.e., the symbol  $\sigma(A(p))\in C(\partial S^*M)\rtimes G$  is invertible, then this family is Fredholm for all  $p\in L$  and invertible at infinity.
\end{corollary}

\paragraph{Symbols of $G$-operators on manifold $M$.}  

Consider a  $G$-operator equal to a finite sum
\begin{equation}\label{eq-ds1}
D=\sum_h D_h \widetilde T_h:L^{2 }(M)\longrightarrow L^{2 }(M)
\end{equation}
on $M$. Let us define its symbol  $\sigma(D)=(\sigma_0(D),\sigma_c(D))$  as the pair
\begin{equation}\label{eq-simbol1}
 \Bigl(\{\sigma_0(D_g)\}, \sum_h\sigma_c(D_h)\widetilde T'_h e^{ p\beta_h}     \Bigr)\in C(S^*M)\rtimes G \oplus \Psi_p^G(\Omega).
\end{equation}
Here the interior symbol is denoted by $\sigma_0(D)$ and the conormal symbol by $\sigma_c(D)$. 

The algebra of symbols   is defined as a subalgebra in the direct sum 
$$
\Sigma\subset C(S^*M)\rtimes G \bigoplus \Psi_p^G(\Omega),
$$
defined by the compatibility condition
$$
\Sigma=\Bigl\{(a_0,a_c)\;\Bigl|\; a_0\bigl|_{\partial S^*M}=\sigma(a_c)\Bigr\}.
$$
Here  $a_0|_{\partial S^*M}\in C(\partial S^*M)\rtimes G$ stands for the restriction of the interior symbol to the boundary, and  $\sigma(a_c)\in C(\partial S^*M)\rtimes G$ is the symbol of the parameter-dependent family  (see above).

Now we can define the symbol mapping as the homomorphism of algebras:
$$
\sigma:\Psi(M)\rtimes G \longrightarrow \Sigma, 
$$
which extends the mapping \eqref{eq-simbol1}. This symbol mapping is well defined by the universal property of the crossed product. 

\paragraph{Quantization of symbols.}

Let us define the quantization mapping
\begin{equation}\label{eq-quant1}
Op: \Sigma  \longrightarrow \mathcal{B}/\mathcal{K},
\end{equation}
which takes a symbol to a certain operator.  

To construct the mapping \eqref{eq-quant1}, consider an element $a=(a_0,a_c)\in\Sigma$, whose first component lies in the algebraic crossed product. Since the elements $(a_0)_g$ are nonzero only for finitely many diffeomorphisms 
$g\in G$,  there exist neighborhoods  $U,V$ of the conical point such that $g(U)\subset V$ and in $U$ all such diffeomorphisms have the form  \eqref{eq-diff1}. Let  $\varphi,\psi$ be smooth functions on $M$, which are identically equal to zero outside  $U$ and $V$, respectively, and are identically equal to 1 in a smaller neighborhood of the conical point.

The symbol $a$ is quantized  in two steps.

First, we use the Mellin transform $\mathcal{M}_{r\to p}$ to define the operator  
$$
\widehat{a}_c=\psi a_c\left( -r\frac{\partial}{\partial r} \right)\varphi\equiv
\psi \mathcal{M}^{-1}_{p\to r} a_c(p)\mathcal{M}_{r\to p}\varphi 
$$
with conormal symbol equal to $a_c$ (this fact is verified directly). 

Second, by the compatibility condition, the interior symbol  $a_0$  and the interior symbol $\sigma_0(\widehat{a}_c)$ of the operator  $\widehat{a}_c$ are equal over  $\partial S^*M$. Hence
\begin{equation}\label{eq-null1}
a_0-\sigma_0\left(\widehat{a}_c\right)\in C_0(S^*M\setminus \partial S^*M)\rtimes G,
\end{equation}
i.e., this difference is a symbol vanishing over $\partial S^*M$. Therefore, to the interior symbol \eqref{eq-null1} we naturally assign an operator on  $M$ with  zero  conormal symbol. Denote this operator by $Op\bigl(a_0-\sigma_0\left(\widehat{a}_c\right)\bigr)$. 
We finally define
\begin{equation}\label{eq-opa1}
Op(a)=Op\bigl(a_0-\sigma_0\left(\widehat{a}_c\right)\bigr)+\widehat{a}_c\in \mathcal{B}/\mathcal{K}.
\end{equation}
It is easy to show that the operator  $Op(a)$ modulo compact operators does not depend on the choice of cut-off functions  $\varphi$ and $\psi$.

\paragraph{Main theorem.}

\begin{theorem}\label{th-cstar1}
\begin{enumerate}
\item The mapping $Op$ (see \eqref{eq-opa1}) extends by continuity to a homomorphism of  $C^*$-algebras and one has a commutative triangle of $C^*$-algebras and their homomorphisms
\begin{equation}\label{eq-triang1}
 \xymatrix{
  \Psi(M)\rtimes G \ar[rr]^\pi \ar[dr]_{\sigma} & & \mathcal{B}/\mathcal{K}\\
   & \Sigma\ar[ru]_{Op}&  
  }
\end{equation}
\item (Fredholm property) Suppose that an element  $A=\{A_g\}\in \Psi(M)\rtimes G$ has invertible symbol $\sigma(A)\in \Sigma$.  Then the corresponding $G$-operator
$$
 \sum_{h\in G} A_h \widetilde T_h:L^{2 }(M)\to L^{2 }(M)
$$ 
is Fredholm.
\end{enumerate}
\end{theorem}
\begin{proof}
1. Let us first prove that the diagram  \eqref{eq-triang1} is commutative in the algebraic case, i.e., given   $A=\{A_g\}\in \Psi(M)\rtimes_{alg} G$  with symbol $\sigma(A)$, let us show that  $\pi(A)=Op(\sigma(A))$.

On the one hand, we have
\begin{equation}
\pi(A)=\sum_h A_h\widetilde T_h.
\end{equation}
On the other hand, the symbol of $A$ is equal to 
$$
\sigma(A)=
\left(
  \{\sigma_0(A_h)\},\sum_h \sigma_c(A_h)(p)\widetilde T'_h e^{ p\beta_h}
\right),
$$ 
where $\sigma_0(A_g)$, $\sigma_c(A_g)$ are the interior and the conormal symbols of  $A_g$. Quantization of this symbol gives the operator 
\begin{multline}\label{eq-long1}
Op(\sigma(A))=\\
=\sum_h Op\left(\sigma_0(A_h)-\sigma_0\left(\psi \sigma_c(A_h)\left(-r\frac\partial{\partial r}\right)  h^*  \varphi\right)  \right)\widetilde T_h+ \sum_h \psi \sigma_c(A_h)\left(-r\frac\partial{\partial r}\right)\widetilde T_h  \varphi=\\
= \sum_h \left[   Op\left(\sigma_0(A_h )-\sigma_0\left(\psi \sigma_c(A_h)\left(-r\frac\partial{\partial r}\right)  h^*  \varphi\right)  \right)+\psi \sigma_c(A_h)\left(-r\frac\partial{\partial r}\right)  h^*  \varphi\right] \widetilde T_h=\\
=\sum_h A_h \widetilde T_h.
\end{multline}
Here the last equality modulo compact operators follows from the fact that the operator in square brackets in  \eqref{eq-long1} and the operator $A_g$ have equal interior and conormal symbols (this is proved by a direct computation) and, hence, coincide up to compact operators.

2. Let us now show that  $Op$ is a homomorphism on the subalgebra $\Sigma_{alg}$, consisting of symbols  $a=(a_0,a_c)$, where $a_0$ is in the algebraic crossed product. Indeed, given $a,b\in \Sigma_{alg}$,
we have  $a=\sigma(A)$, $b=\sigma(B)$ for some elements $A,B\in \Psi(M)\rtimes_{alg} G$
(since the symbol mapping  $\sigma$ is surjective). Further, we have
$$
Op(ab)=Op(\sigma(A)\sigma(B))=Op(\sigma(A B))=\pi(AB)=\pi(A)\pi(B)= Op(a)Op(b).
$$
This gives the desired property.

3.  It follows from  item 2. that  $Op$   extends by the universal property of the crossed product from the dense subalgebra  $\Sigma_{alg}\subset \Sigma$ to the entire algebra $\Sigma$. For this extension the diagram  \eqref{eq-triang1} remains well defined and commutative. 

4. To prove the Fredholm property, it suffices to show that  $\pi(A)$ is invertible. We claim that the inverse element is equal to  $Op(\sigma(A)^{-1})$. Indeed, we have
$$
\pi(A)\cdot Op(\sigma(A)^{-1})= Op(\sigma(A))Op(\sigma(A)^{-1})=Op(1)=1.
$$
Here the first equality follows from the commutativity of diagram  \eqref{eq-triang1}, while the second follows from the fact that  $Op$ is a homomorphism of algebras. 

The equality
$$
Op(\sigma(A)^{-1})\cdot \pi(A)=  1
$$
is proved similarly.

The proof of the theorem is now complete.
\end{proof}

\paragraph{Invertibility of symbols.}

By Theorem~\ref{th-cstar1}  a $G$-operator is Fredholm, if its symbol is invertible as an element of a certain $C^*$-algebra. How to check this invertibility? The following proposition shows that the invertibility condition in Theorem~\ref{th-cstar1} can be equivalently written as invertibility of certain explicitly written operators. To write out these conditions, we define a special set of representations
$$
\begin{array}{ccc}
   \pi_{x,\xi}:C(S^*M)\rtimes G &\lra &\mathcal{B}l^2(G)\vspace{2mm}\\
   f &\longmapsto & \sum_h f(g^{-1}(x),\partial g^{-1}(\xi),h)\mathcal{T}_h
\end{array}
$$
of the crossed product in the algebra of bounded operators acting on the standard space $l^2(G)$ on the group (cf. \eqref{traj-symbol1}). Given  $(x,\xi)\in S^*M$, this representation takes an element of the crossed product to its restriction on the orbit of   $(x,\xi)$.  
\begin{proposition}\label{prop-5}
The symbol 
$$
 \sigma(A)=(\sigma_0(A),\sigma_c(A))\in C(S^*M)\rtimes G \bigoplus \Psi_p^G(\Omega)
$$ 
is invertible if and only if the following two conditions are satisfied:
\begin{enumerate}
\item[1)] the restriction  $\pi_{x,\xi}(\sigma_0(A)):l^2(G)\to l^2(G)$ of the  interior symbol on the orbit  of an arbitrary point  $(x,\xi)\in S^*M$ is invertible;
\item[2)]  the conormal symbol $\sigma_c(A)(p): L^2(\Omega)\to L^2(\Omega)$ is invertible for all $p\in L$.
\end{enumerate}
\end{proposition}
\begin{proof}
Since  $\sigma(A)$ is a pair, it suffices to prove the invertibility of each of the components. 

1. The interior symbol  $\sigma_0(A)$ is in a crossed product. By a well-known result \cite{AnLe1} an element of the crossed product is invertible if and only if its restrictions $\pi_{x,\xi}(\sigma_0(A))$ to the orbits of all points $(x,\xi)\in S^*M$ are invertible. Hence, by our assumptions the interior symbol is invertible in the crossed product.

2. Let us prove the invertibility of the conormal symbol $\sigma_c(A)(p)\in C(L,\mathcal{B}L^2(\Omega))$. By assumption, this symbol is invertible for each fixed $p$. We have to show that the norm of the inverse family is uniformly bounded  (continuity of the inverse family is obvious). Now uniform boundedness on an arbitrary finite interval of $L$ follows from continuity, while for large values of the parameter uniform boundedness follows from the fact that the symbol $\sigma(\sigma_c(A))\in C(\partial S^*M)\rtimes G$ is invertible, i.e. the parameter-dependent family  $\sigma_c(A)(p)$ is elliptic and by Corollary~\ref{cor-7} the inverse family is uniformly bounded for large  $p$.

The proof of the proposition is complete. 
\end{proof}

\paragraph{Proof of Theorem~\ref{th-finite1} (Fredholm property in weighted Sobolev spaces $H^{s,\gamma}$).}

Let  
$$
  D:H^{s,\gamma}(M)\to H^{s-m,\gamma+m}(M)
$$ 
be an elliptic $G$-operator of order  $m$ on  $M$. We want to prove that it is Fredholm. 

1. (Reduction to operator on $L^2$.)  First, we reduce $D$ to a zero-order operator acting on  $L^2(M)$. To this end,   consider the composition 
$$
L^{2}(M)\stackrel{D_1}\longrightarrow H^{s,\gamma}(M) \stackrel{D}\longrightarrow H^{s-m,\gamma+m}(M) \stackrel{D_2}\longrightarrow L^{2}(M), 
$$
where  $D_1,D_2$ stand for some elliptic $\psi$DOs on  $M$, acting in corresponding weighted spaces.  This composition is a $G$-operator  denoted by 
$$
 D'=D_2DD_1.
$$  
Since   $D_{1,2}$ are elliptic and by classical conical theory Fredholm, the $G$-operators  $D$ and $D'$ are Fredholm or not Fredholm simultaneously. Hence, by the multiplicative property of the symbols we get
$$
\sigma(D')=\sigma(D_2)\sigma(D)\sigma(D_1).
$$
Hence, ellipticity of  $D$ implies ellipticity of $D'$. 

2. Let us show that ellipticity of operator $D'$ in the sense of Definition~\ref{def-ell4} implies invertibility of its symbol in the algebra of symbols  
$$
\sigma(D')\in C(S^*M)\rtimes G \bigoplus  C(L,\mathcal{B}L^2(\Omega)),\qquad L=\{\re p=0\}.
$$ 
Indeed, invertibility  conditions for such elements were obtained in Proposition~\ref{prop-5}. Let us show, that the conditions of this proposition are satisfied. Indeed, the conormal symbol  $\sigma_c(D)(p)$ is invertible by assumption for all $p\in L$. It remains to prove the invertibility of the restriction $\pi_{x,\xi}(\sigma_0(D))$ to the orbit of an arbitrary point $(x,\xi)\in S^*M$. Consider the diagram
\begin{equation}\label{eq-diag1}
\xymatrix{ l^2(G,m_{x,\xi })\ar[rr]^{\sigma(D')(x,\xi)} \ar[d]^\simeq_{\sqrt{m_{x,\xi }}}& &
l^2(G,m_{x,\xi }) \ar[d]^\simeq_{\sqrt{m_{x,\xi }}}\\
l^2(G) \ar[rr]_{\pi_{x,\xi}(\sigma (D'))} & & l^2(G),}
\end{equation}
where the vertical isomorphisms are products with the square root of the weight function 
 $m_{x,\xi }(g)$ in \eqref{measure-invariant1} (we set $s=\gamma=0$).  This diagram is commutative (this fact is easy to prove directly if  we write operator  $D'$ as in~\eqref{eq-ds1}). Now, since  $D'$ is elliptic in the sense of Definition~\ref{def-ell4}, it follows that the upper row in~\eqref{eq-diag1} is an isomorphism. Since the diagram commutes, we obtain that the lower row is an isomorphism as well. Hence, the restriction $\pi_{x,\xi}(\sigma_0(D'))$ of the symbol to the orbit of  $x,\xi$ is invertible as desired. 

The proof of Theorem~\ref{th-finite1}  is now complete.

\section{Appendix. Computation of the interior symbol}

{Let us give the details of the computation of the interior  symbol \eqref{traj-symbol1} and the weight \eqref{measure-invariant1} using the method of frozen coefficients. 

Let us consider a differential operator $D$ on $M$,  and $x_0\in M\setminus pt$. By freezing the coefficients of $D$ at the points of the orbit  $Gx_0\subset M$ and omitting lower order terms, we obtain an operator with constant coefficients in the space of functions on the orbit:\footnote{Recall that the norm in the Sobolev space  $H^s(V,|\cdot|,d\xi)$, where $V$ is a vector space, is defined by a norm $|\cdot|$ and measure $d\xi$ on $V^*$ by the formula
$$
\|f\|^2_s=\int_V  (1+|\xi|^2)^s|\widetilde f(\xi)|^2 d\xi,\text{where $\widetilde f(\xi)$ is the Fourier transform of $f(x)$}.
$$}
\begin{multline*}
\bigoplus_{g\in G } \sigma_0(D)\left(g^{-1}(x_0),-i\frac{\partial}{\partial x_g}\right):\bigoplus_g H^s\Bigl(T_{g^{-1}(x_0)}M,|\cdot|_{g^{-1}(x_0)},\mu_{\gamma}(g^{-1}(x_0))\Bigr)\longrightarrow\\
\bigoplus_g H^{s-m}\Bigl(T_{g^{-1}(x_0)}M,|\cdot|_{g^{-1}(x_0)},\mu_{\gamma+m}(g^{-1}(x_0))\Bigr).
\end{multline*}
Fourier transform $x_g\to \xi_g$ at each point of the orbit takes this operator to an isomorphic operator:\footnote{Here and below $L^2(X,\mu_X)$ is the  $L^2$-space of functions on a space $X$ square integrable with respect to the measure $\mu_X$.}
\begin{multline}\label{eq-tempo2}
\bigoplus_{g\in G } \sigma_0(D)\left(g^{-1}(x_0),\xi_g\right):\bigoplus_g L^2\Bigl(T^*_{g^{-1}(x_0)}M,(1+|\xi_g|^2)^s\mu_\gamma(g^{-1}(x_0))\Bigr)\longrightarrow\\
\bigoplus_g L^2\Bigl(T^*_{g^{-1}(x_0)}M,(1+|\xi_g|^2)^{s-m}\mu_{\gamma+m}(g^{-1}(x_0))\Bigr).
\end{multline}
As usual in elliptic theory, the symbol \eqref{eq-tempo2} is considered off the zero section, i.e., we replace the operator   \eqref{eq-tempo2}  by
\begin{multline}\label{eq-opl2}
\bigoplus_{g\in G } \sigma_0(D)\left(g^{-1}(x_0),\xi_g\right):\bigoplus_g L^2\Bigl(T^*_{g^{-1}(x_0)}M, |\xi_g|^{2s}\mu_\gamma(g^{-1}(x_0))\Bigr)\longrightarrow \\
\bigoplus_{g\in G } L^2\Bigl(T^*_{g^{-1}(x_0)}M, |\xi_g|^{2(s-m)}\mu_{\gamma+m}(g^{-1}(x_0))\Bigr).
\end{multline}

Operator \eqref{eq-opl2} acts on a direct sum of spaces, corresponding to different points of  $M$. This is not quite convenient from the point of view of computations. To eliminate this drawback, we replace   \eqref{eq-opl2} by an isomorphic operator acting on the space of functions at the point  $x_0$. To this end, we note that for   $g\in G$ the codifferential of  $g:M\to M$ is an isomorphism of  cotangent spaces
$$
\partial g:T^*_{g^{-1}(x_0)}M \longrightarrow T^*_{x_0}M.
$$ 
Hence, the corresponding shift operator is invertible and defines an isometric isomorphism of $L^2$-spaces:\footnote{Here we use the following elementary fact: if $\varphi:X\to Y$ is a diffeomorphism, then the mapping  $\varphi^*:L^2(Y,\mu_Y)\to L^2(X,\mu_X)$ defines an isometric isomorphism if and only if $\mu_X=\varphi^*\mu_Y$.}
\begin{equation}\label{eq-op22}
T_{\partial g}^{-1}=(\partial g)^*: L^2\Bigl(T^*_{ x_0 }M, |\xi|^{2s}\mu_\gamma(x_0) m_{s,\gamma}(g)\Bigr)\longrightarrow
L^2\Bigl(T^*_{ g^{-1}(x_0) }M,|\xi_g|^{2s}\mu_{\gamma }(g^{-1}(x_0))\Bigr),
\end{equation}
where the weight function $m_{s,\gamma}(g)$ on  $T^*_{x_0}M$  is defined as
$$
m_{s,\gamma}(g)= \frac{\Bigl({\partial g^{-1}}^{*}(\mu_\gamma |\xi|^{2s})\Bigr)(x_0,\xi)}{\mu_\gamma(x_0)|\xi|^{2s}} ,
$$
which coincides with  \eqref{measure-invariant1} up to invertible factor  $|\xi|^{2s}$.
Using isometric isomorphisms \eqref{eq-op22},  the spaces in  \eqref{eq-opl2} can be replaced by isomorphic subspaces of the form
\begin{equation}\label{eq-isom4}
\begin{array}{ccc}
\bigoplus_{g\in G } T_{\partial g} :\bigoplus_g L^2\Bigl(T^*_{g^{-1}(x_0)}M,|\xi|^{2s}\mu_\gamma(g^{-1}(x_0))\Bigr) & \stackrel\simeq\longrightarrow&
L^2\Bigl(T^*_{x_0}M, |\xi| ^{2s}\mu_{\gamma}(x_0); l^2(G,m_{s,\gamma})\Bigr), \vspace{2mm}\\
\{w(\xi_g)\} & \longmapsto & \{w(\partial g^{-1}(\xi))\},
\end{array}
\end{equation}
where the mapping ranges in the space of functions on  $T^*_{x_0}M$ with values in the space $l^2(G,m_{s,\gamma})$ of functions on the group.    

So, taking into account \eqref{eq-opl2} and \eqref{eq-isom4}, we see that the symbol of $D$, which is obtained by freezing the coefficients of the operator on the orbit and applying the Fourier transform, can be represented as the family of operators of multiplication 
\begin{equation}\label{eq-18}
  \begin{array}{cccc}
      \sigma_0(D)(x_0,\xi): & l^2(G,m_{s,\gamma})& \longrightarrow & l^2(G,m_{s-m,\gamma+m}), \vspace{2mm}\\
        & w(g) & \longmapsto &  \sigma_0(D)\bigl(g^{-1}(x_0), \partial g^{-1}(\xi)\bigr)w(g)
  \end{array}
\end{equation}
by the values of the interior symbol at the corresponding points on the orbit of  $(x_0,\xi)\in T^*_0M$.

A similar computation (which we omit for brevity) shows that the symbol of the shift operator  $T_h$ is equal to the right shift operator on the group:
\begin{equation}\label{eq-19}
  \begin{array}{cccc}
      \sigma_0(T_h)(x_0,\xi)=\mathcal{T}_h: & l^2(G,m_{s,\gamma})& \longrightarrow & l^2(G,m_{s ,\gamma }), \vspace{2mm}\\
      &   w(g) & \longmapsto &  w(gh).
  \end{array}
\end{equation}
Formulas \eqref{eq-18} and  \eqref{eq-19} give the desired formula for the symbol \eqref{traj-symbol1}, and also the formula  \eqref{measure-invariant1} for the weight function.
}


\addcontentsline{toc}{section}{References}


\end{document}